\documentclass[12pt]{amsart}
\usepackage{amssymb,amsmath,amsthm,amssymb,amsfonts,amscd}
\usepackage{graphicx}
\graphicspath{ {images/} }
\usepackage{bm}
\usepackage{color, enumerate}
\usepackage[all]{xy}
\usepackage{bigints}
\usepackage{enumitem}
\usepackage[english]{babel}
\usepackage{dsfont} 
\usepackage{hyperref}
\usepackage{verbatim}
\usepackage[mathscr]{euscript}
\usepackage{enumitem}

\usepackage{soul}

\newcommand{\rr}{\mathbb R}

\newcommand{\calF}{\mathscr F}

\renewcommand{\epsilon}{\varepsilon}

\newcommand{\lra}{\longrightarrow}

\newcommand{\abs}[1]{\lvert#1\rvert}
\newcommand{\norm}[1]{\lVert#1\rVert}
\newcommand{\Lspace}[1]{L^{#1}(\rr)}

\definecolor{UMRed}{rgb}{0.73,0.09,0.19}   

\newcommand{\ds}{\displaystyle}

\DeclareMathOperator*\essinf{ess \ inf} 
\newcommand{\LH}{\operatorname{LH}}

\newcommand{\Fourier}{\calF}

\newcommand{\Schwartz}{\mathcal{S}}

\newcommand{\pinfMr}{p_{\infty, r}^+}
\newcommand{\pinfMl}{p_{\infty, l}^+}
\newcommand{\pinfMstar}{p_{\infty, *}^+}

\newcommand{\pinfmr}{p_{\infty, r}^-}
\newcommand{\pinfml}{p_{\infty, l}^-}
\newcommand{\pinfmstar}{p_{\infty, *}^-}

\newcommand{\qinfMr}{q_{\infty, r}^+}
\newcommand{\qinfMl}{q_{\infty, l}^+}
\newcommand{\qinfMstar}{q_{\infty, *}^+}

\newcommand{\qinfml}{q_{\infty, l}^-}
\newcommand{\qinfmr}{q_{\infty, r}^-}
\newcommand{\qinfmstar}{q_{\infty, *}^-}
\newcommand{\pinfmrtheta}{p_{\infty, \theta}^-}

\newcommand{\qinfMrtheta}{q_{\infty, \theta}^+}

\theoremstyle{plain}
\newtheorem{theorem}{Theorem}[section]
\newtheorem{corollary}[theorem]{Corollary}
\newtheorem{lemma}[theorem]{Lemma}
\newtheorem{proposition}[theorem]{Proposition}

\theoremstyle{definition}

\newtheorem{remark}[theorem]{Remark}

\usepackage[
    backend=biber,
    style=numeric
]{biblatex}

\renewbibmacro{in:}{} 
\usepackage{etoolbox}
\patchcmd{\thmhead}{(#3)}{#3}{}{}

\DeclareMathAlphabet{\mathpzc}{OT1}{pzc}{m}{it}

\newcommand{\N}{\mathbb{N}}

\def\supp{{\rm supp}}

\def\N{{\mathbb{N}}}

\def\0{{\rm \bf{0}}}

\def\B{{\mathcal{B}}}

\DeclareMathOperator*{\esssup}{ess\,sup}

\def\B2star{\overline{B}_X^{w(X^{\ast\ast},X^{\ast})}}

\title{Fourier inequalities in variable Lebesgue spaces}\thanks{This work was partially supported by CONICET PIP  11220200102366CO, 
CONICET PIP 11220200101609CO, 
SIIP 80020240100196UN 
UNCOMA PIN I 04/B251 
and UNCuyo 06/80020240100069UN 
}

\subjclass[2010] {}

\date{}

\begin{document}
\baselineskip=.65cm

\author[Cardeccia]{Rodrigo Cardeccia}
\address[Cardeccia]{Dpto. de Matemática Instituto Balseiro, CNEA-Universidad Nacional de Cuyo, CONICET, San Carlos de Bariloche, Argentina.}
\email{\texttt{rodrigo.cardeccia@ib.edu.ar}}

\author[Caruso]{Mat\'ias I. Caruso}
\address[Caruso]{Dpto. de Matemática Instituto Balseiro, CNEA-Universidad Nacional de Cuyo, CONICET, San Carlos de Bariloche, Argentina.}
\email{\texttt{matias.caruso@ib.edu.ar}}

\author[Mazzitelli]{Martin Mazzitelli}
\address[Mazzitelli]{Dpto. de Matemática Instituto Balseiro, CNEA-Universidad Nacional de Cuyo, CONICET, San Carlos de Bariloche, Argentina.}
\email{\texttt{martin.mazzitelli@ib.edu.ar}}

\author[Rodr\'{i}guez]{Jorge Tom\'as Rodr\'{i}guez}
\address[Rodr\'{i}guez]{NuCoMPA, Facultad de Cs. Exactas, Universidad Nacional del Centro de la Provincia de Buenos Aires, CONICET, Tandil, Argentina.
    }
\email{\texttt{jtrodriguez@nucompa.exa.unicen.edu.ar}}

\begin{abstract}
We study the boundedness of the Fourier transform on variable Lebesgue spaces. We obtain necessary conditions and, independently, sufficient conditions on the exponents $p(\cdot)$ and $q(\cdot)$ under which the inequality $$\|\hat{f}\|_{q(\cdot)}\leq C\|f\|_{p(\cdot)}$$
holds for some constant $C>0$ and all $f\in L^{p(\cdot)}(\mathbb{R})$. As a byproduct, we improve some recent results of Saucedo and Tikhonov in \cite{ar:saucedo_tikhonov:2025:note_on_fourier_inequalities}. Moreover, when $p(x)\to 1$ and $q(x)\to +\infty$ as $|x|\to +\infty$, we characterize the exponents for which the Fourier transform is bounded.
\end{abstract}

\maketitle

\section{Introduction}

Given a Schwartz function $f\colon \mathbb{R}\to \mathbb{C}$, its Fourier transform is defined by
\begin{equation*}\label{TF}
\hat{f}(\xi) := \int_{\mathbb{R}} f(x)e^{-2\pi i x \xi}\, dx, \quad \xi\in \mathbb{R}.
\end{equation*}
This definition extends to functions in $L^p(\mathbb{R})$ with $1\leq p\leq 2$, where the classical Hausdorff--Young inequality
\begin{equation}\label{HY}
\|\hat{f}\|_{p'}\leq \|f\|_p
\end{equation}
holds for every $f \in L^p(\rr)$ and $\tfrac{1}{p}+\tfrac{1}{p'}=1$.
In the setting of variable exponent Lebesgue spaces, one might naturally conjecture that   
\begin{equation}\label{HYfail}
\|\hat{f}\|_{p'(\cdot)}\leq \|f\|_{p(\cdot)}, 
\end{equation}
whenever $1\leq p(x)\leq 2$ and $\tfrac{1}{p(x)}+\tfrac{1}{p'(x)}=1$ for all $x\in \mathbb{R}$. However, as shown in \cite[Section~5.6.10]{bo:cruz-uribe_fiorenza:variable_lebesgue_spaces}, taking $f(x) := |x|^{-2/3}$ and a smooth exponent $p(\cdot)\colon \mathbb{R}\to [1,\infty]$ such that $p(x)=5/4$ on $(-1,1)$ and $p(x)=2$ on $\mathbb{R}\setminus (-2,2)$, one has $f\in L^{p(\cdot)}(\mathbb{R})$ but $\hat{f}\notin L^{p'(\cdot)}(\mathbb{R})$. This simple example reveals that, if a generalization of the Hausdorff--Young inequality exists, it must involve some \emph{extra ingredient} not accounted for in \eqref{HYfail}. Motivated by some weighted Fourier inequalities in \cite{benedetto1987fourier}, Cruz--Uribe and Fiorenza posed the question of whether the inequality
\[
\|\hat{f}\|_{p'\left( \frac{1}{x} \right)}\leq C \|f\|_{p(x)}
\]
holds for some constant $C>0$, where $p(\cdot)$ is an even, non-decreasing exponent on $(0,+\infty)$ satisfying $1\leq p(\cdot) \leq 2$ (see \cite[Question~5.61]{bo:cruz-uribe_fiorenza:variable_lebesgue_spaces}). This question was recently answered in the negative by Saucedo and Tikhonov in \cite{ar:saucedo_tikhonov:2025:note_on_fourier_inequalities}. They proved that the above inequality holds only in the constant exponent case under the additional assumptions that $p(\cdot)$ is continuous, has a limit at infinity, and satisfies a certain regularity condition. 

Our main objective is to undertake a systematic study of the validity of Hausdorff--Young type inequalities in variable exponent Lebesgue spaces. In Theorem~\ref{thm: necessary and sufficient thm} we show that, in a certain sense, the behavior of $p(\cdot), q(\cdot)\in \mathcal{P}(\rr)$ at infinity is crucial for the boundedness of $\Fourier\colon L^{p(\cdot)}(\rr)\to L^{q(\cdot)}(\rr)$. On the one hand, the embeddings $L^{p(\cdot)}(\mathbb{R})\hookrightarrow L^{p_-}(\mathbb{R})$ and $L^{(p_-)'}(\mathbb{R})\hookrightarrow L^{q(\cdot)}(\mathbb{R})$ together with the conditions $p_-\leq 2$ and $q_+\leq (p_-)'$ are sufficient for this boundedness (we refer the reader to the precise definitions and notation given below). On the other hand, the boundedness of $\Fourier\colon L^{p(\cdot)}(\rr)\to L^{q(\cdot)}(\rr)$ implies that $(p_-)'=q_+\geq 2$, and that the \emph{essential limit inferior} (respectively, \emph{superior}) of $p(\cdot)$ (respectively, $q(\cdot)$) coincides with $p_-$ (respectively, $q_+$). We go further in Theorem~\ref{Thm: equivalencia q_+=infinito}, where we obtain the desired characterization in the particular case
$$
\lim_{|x| \to +\infty} p(x) = 1 \qquad \text{and} \qquad \lim_{|x| \to +\infty} q(x) = +\infty.
$$
We also study \emph{how large} the sets $\{x\in \rr : p(x) > 2+\varepsilon\}$ and $\{x\in \rr : q(x) < 2-\varepsilon\}$ can be when $\Fourier\colon L^{p(\cdot)}(\rr)\to L^{q(\cdot)}(\rr)$ is bounded.

\subsection{Preliminaries}
Variable exponent Lebesgue spaces were introduced by Orlicz in the early 30's and systematically studied by Nakano \cite{bo:nakano:modulared_semi-ordered_linear_spaces,bo:nakano:topology_of_linear_topological_spaces} in the 50's.  
In recent years, these spaces have attracted significant attention due to their applications in the modeling of \emph{electrorheological fluids} and in \emph{image processing} (see \cite{bo:diening_harjulehto_hasto_ruzicka:Lebesgue_and_Sobolev_spaces_with_variable_exponents} and the references therein). This has motivated the study of various properties of these spaces, in particular those already known in the setting of classical Lebesgue spaces, where the exponent $p$ is constant. At the same time, many authors have devoted efforts to developing harmonic analysis in the variable exponent setting: in this direction, we may mention \cite{ar:cruz-uribe_diening_fiorenza:2009:a_new_proof_of_the_boundedness_of_maximal_operators_on_variable_Lebesgue_spaces, ar:cruz-uribe_fiorenza_martell_perez:2006:the_boundedness_of_classical_operators_on_variable_Lp_spaces, ar:cruz-uribe_fiorenza_neugebauer:2003:the_maximal_function_on_variable_Lp_spaces, ar:diening:2004:maximal_function_on_generalized_Lebesgue_spaces_Lp, ar:nekvinda:2004:Hardy-Littlewood-maximal_operator_on_LpRn}, where the boundedness of the Hardy--Littlewood maximal operator and other classical operators of harmonic analysis on $L^{p(\cdot)}(\mathbb{R}^n)$ is investigated.

We start by fixing the notation and recalling the definitions that are essential in what follows.
Given a $\sigma$-finite complete measure space $(\Omega, \Sigma, \mu)$, we denote by $\mathcal{P}(\Omega,\mu)$ the set of all $\mu$-measurable functions (exponents) $p(\cdot)\colon \Omega \to [1,+\infty]$. As usual,
\begin{equation}\label{def_p+p-}
p_{_-} := \essinf_{x\in \Omega}p(x) \qquad \text{and} \qquad p_{_+} := \esssup_{x\in \Omega}p(x),
\end{equation}
while $\mathcal{P}_b(\Omega,\mu)=\{p\in \mathcal{P}(\Omega,\mu):\,\, 1<p_-\leq p_+<+\infty\}$ denotes the set of all bounded exponents. Given an exponent function $p(\cdot)$, let $\rho_{p(\cdot)}$ be the \emph{modular functional} associated to $p(\cdot)$ given by
$$
\rho_{p(\cdot)}(f) := \int_{\Omega\setminus \Omega_\infty}|f(x)|^{p(x)} \ d\mu(x) + \|f\|_{L^\infty \left( \Omega_\infty \right)},
$$
where $\Omega_\infty := \{x\in \Omega:\,\, p(x)=+\infty\}$. We denote by $L^{p(\cdot)}(\Omega, \mu)$ the Banach space of $\mu$-measurable functions $f$ such that $\rho_{p(\cdot)}(f/\lambda)<\infty$ for some $\lambda>0$, endowed with the Luxemburg norm
$$
\|f\|_{L^{p(\cdot)}(\Omega, \mu)} := \inf\{\lambda > 0 \mid \rho_{L^{p(\cdot)}(\Omega)}(f/\lambda)\leq 1\}.
$$
Since we will work on $\rr$ equipped with the Lebesgue measure, we shall simply write $\mathcal P(\mathbb R)$, $\mathcal P_b(\mathbb R)$ and $L^{p(\cdot)}(\rr)$.
If $E\subseteq \rr$ is measurable, we denote by $L^{p(\cdot)}(E)$ the corresponding variable exponent Lebesgue space of measurable functions on $E$.

To study the behavior of exponent functions $p(\cdot)\in\mathcal P(\mathbb R)$ at infinity, in what follows we will consider the \emph{essential limit inferior} and \emph{essential limit superior} at $+\infty$ and $-\infty$ defined by
\[
\pinfmr := \lim_{M \to +\infty} \essinf_{x \ge M} p(x), \quad \pinfMr := \lim_{M \to +\infty} \esssup_{x \ge M} p(x)
\]
and
\[
\pinfml := \lim_{M \to -\infty} \essinf_{x \le M} p(x), \quad \pinfMl := \lim_{M \to -\infty} \esssup_{x \le M} p(x).
\]
Finally, we will denote by $C_0(\rr)$ the space of continuous functions vanishing at infinity, by $C^\infty_c(\rr)$ the space of infinitely differentiable functions with compact support, and by $\Schwartz(\rr)$ the Schwartz space on $\rr$. It is clear that $C^\infty_c(\rr)\subseteq \Schwartz(\rr)$ and that both spaces are dense in $L^p(\rr)$ when $1\leq p<\infty$. Moreover, if $p(\cdot)\in \mathcal{P}_b(\rr)$ then $C^\infty_c(\rr)$ and $\Schwartz(\rr)$ are dense in $L^{p(\cdot)}(\rr)$ (see \cite[Theorem~3.4.12]{bo:diening_harjulehto_hasto_ruzicka:Lebesgue_and_Sobolev_spaces_with_variable_exponents}). 

Occasionally, we will use the notation $\lesssim$ to denote an inequality up to a multiplicative constant. 

\subsection{Main results}
The starting point of our analysis is one of the main results of this paper.

\begin{theorem}\label{thm: necessary and sufficient thm}
Let $p(\cdot), q(\cdot)\in \mathcal{P}(\mathbb{R})$. Consider the following statements.
\begin{enumerate}
\item The conditions $p_-\leq 2$, $q_+\leq (p_-)'$ and $1\in L^{r(\cdot)}(\mathbb{R})\cap L^{s(\cdot)}(\mathbb{R})$ hold, where $r(\cdot), s(\cdot)$ are the defect exponents defined by
$$
\frac{1}{r(\cdot)}=\frac{1}{p_-}-\frac{1}{p(\cdot)} \quad \text{and} \quad \frac{1}{s(\cdot)}=\frac{1}{q(\cdot)}-\frac{1}{(p_-)'}.
$$
\item The Fourier transform $\Fourier : L^{p(\cdot)}(\mathbb{R})\to L^{q(\cdot)}(\mathbb{R})$ is bounded.
\item The conditions
$$
\pinfmr=\pinfml = p_-\leq 2, \quad \qinfMr=\qinfMl =q_+ \quad  \text{and} \quad (p_-)'=q_+
$$
hold.
\end{enumerate}
Then, the implications $(i)\Rightarrow (ii)\Rightarrow (iii)$ hold.
\end{theorem}
The proof of this theorem relies on  Proposition~\ref{prop: inclusion transformada de Fourier}, which essentially relates the boundedness of $\Fourier \colon L^{p(\cdot)}(\mathbb{R})\to L^{q(\cdot)}(\mathbb{R})$ with the boundedness of $\Fourier : L^{\pinfmr}(\rr)\cap L^{\pinfMr} (\rr)\to L^{q(\cdot)}(\rr)$ and $\Fourier : L^{\pinfml}(\rr)\cap L^{\pinfMl}(\rr) \to L^{q(\cdot)}(\rr)$. It also provides a duality-type principle which, in the case of bounded exponents, asserts that the boundedness of the Fourier transform from $L^{p(\cdot)}(\rr)$ into $L^{q(\cdot)}(\rr)$ is equivalent to that from $L^{q'(\cdot)}(\rr)$ into $L^{p'(\cdot)}(\rr)$ (see also Remark~\ref{rmk: dual de la transf}).

It is worth mentioning that the implication $(ii)\Rightarrow (iii)$ in Theorem~\ref{thm: necessary and sufficient thm} is a refinement of the following recent result of Saucedo and Tikhonov. Moreover, the techniques used in the present work build upon ideas that are closely related to those developed in their approach.
\begin{theorem}{(\cite[Theorem~1.1]{ar:saucedo_tikhonov:2025:note_on_fourier_inequalities})}
    Let $p(\cdot), q(\cdot)\in \mathcal{P}(\rr)$ be such that $\Fourier\colon L^{p(\cdot)}(\rr)\to L^{q(\cdot)}(\rr)$ is bounded. Assume that $p(\cdot)$ is continuous, even, non-decreasing on $(0,+\infty)$, and that $\lim_{x\to+\infty}p(x)=p_\diamond \leq 2$. Then $p(x)\leq p_\diamond\leq q'(x)$ for all $x\in \rr$.
\end{theorem}
\begin{remark}
    It is easy to see that conditions $(i)$ and $(iii)$ in Theorem~\ref{thm: necessary and sufficient thm} are not equivalent. Take, for instance,
    \[
	p(x) :=
	\begin{cases}
		1 &\text{if }  x\in \bigcup_{n=1}^\infty I_n\\
		2 &\text{otherwise},
	\end{cases}
\]
where 
$$
I_n=(n, n+2^{-n}) \cup (-n, -n + 2^{-n}),
$$
and let $q(\cdot)=p'(\cdot)$. Then $p(\cdot)$ and $q(\cdot)$ satisfy $(iii)$, but not $(i)$. Now, it is natural to ask if either $(i)$ or $(iii)$ is equivalent to the boundedness of the Fourier transform. Regarding this question, see Theorem~\ref{Thm: equivalencia q_+=infinito} and Problem~\ref{problem: equivalence} in Section~\ref{problems}. 
\end{remark}
The proof of Theorem~\ref{thm: necessary and sufficient thm}, as well as those of the auxiliary results, is given in Section~\ref{sec necessary and sufficient conditions}. In that section, we also show that the relations obtained for the exponents in Theorem~\ref{thm: necessary and sufficient thm} are, in a certain sense, optimal. More precisely, we prove that there are no further inequalities involving $(p_+)'$, $(\pinfMr)'$, $(\pinfMl)'$, $q_-$, $\qinfmr$ or $\qinfml$, beyond the trivial ones.

In Section~\ref{Seccion caracteristicas}, we carry out a deeper analysis of the conditions for the boundedness of the Fourier transform. We analyze the behavior of the norms of characteristic functions of large intervals.
This approach allows us to sharpen the analysis of the case $q(x)\to+\infty$ and $p(x)\to1$ by proving that condition $(i)$ in Theorem~\ref{thm: necessary and sufficient thm} is equivalent to the boundedness of $\Fourier\colon L^{p(\cdot)}(\rr)\to L^{q(\cdot)}(\rr)$. Concretely, we have the following result.

\begin{theorem}\label{Thm: equivalencia q_+=infinito}
Let $p(\cdot),q(\cdot)\in \mathcal{P}(\rr)$ be such that
\[
\lim_{|x| \to +\infty} p(x)  = 1 \qquad \text{and} \qquad \lim_{|x| \to +\infty} q(x) = +\infty.
\]
Then, $\Fourier: L^{p(\cdot)}(\rr)\to L^{q(\cdot)}(\rr)$ is bounded if and only if $1\in L^{p'(\cdot)}(\rr)\cap L^{q(\cdot)}(\rr)$.
\end{theorem}

Finally, in Section \ref{Seccion lagunares} we study the geometric structure of the sets $\{x\in \rr: p(x)>2+\varepsilon\}$ and $\{x\in \rr: q(x)>2-\varepsilon\}$. In Theorem \ref{thm:lagunares}, we show that these sets cannot contain infinitely many intervals of fixed length. The proof of this result relies on the equivalence of the $L^p$ norms of the so-called \emph{Fourier lacunary series}.

\section{Conditions on the exponents}\label{sec necessary and sufficient conditions}
In this section, we prove Theorem~\ref{thm: necessary and sufficient thm}, which provides necessary and, independently, sufficient conditions on the variable exponents $p(\cdot)$ and $q(\cdot)$ for the boundedness of the Fourier transform between variable Lebesgue spaces.

Before turning to the proof of the theorem, we establish two preliminary results. Under the assumption that $\Fourier : L^{p(\cdot)}(\rr) \to L^{q(\cdot)}(\rr)$ is bounded, the first one asserts that $\pinfmr, \pinfml\leq 2$, while the second essentially establishes the boundedness of the Fourier transform on the intersection of Lebesgue spaces determined by the asymptotic exponents (it also provides a duality-type principle). We present their proofs in a separate subsection.

\subsection{Properties of the Fourier transform on variable Lebesgue spaces}

\begin{proposition}\label{prop: lim menor a 2}
Let $p(\cdot)\in \mathcal{P}(\rr)$ be such that $\hat f$ is a locally integrable function for every $f\in L^{p(\cdot)}(\rr)$. Then, $\pinfmr, \pinfml\leq 2$.
\end{proposition}
\begin{proof}
Suppose that $\pinfmr >2$ and arrive at a contradiction. Take $2<r<\pinfmr$ and $A>0$ such that $p(x) \geq r$ for almost every $x>A.$ Since $r>2$ there is $g\in \Lspace{r}$ such that its distributional Fourier transform $\hat{g}$ is not a locally integrable function (see \cite[Exercise 2.3.13]{bo:grafakos:classical_fourier_analysis}). By restricting the domain and replacing $g$ with $\tilde{g}(x):=g(-x)$ if necessary, we may suppose $\supp (g) \subseteq \rr_{\geq 0}$.
Now, take $$E=\{x>A: |g(x)|<1\} \qquad \text{and}\qquad F=\rr_{\geq 0}\setminus E,$$ and write $g=g_1+g_2$, where $g_1:=\chi_E g$ and $g_2:=\chi_{F} g$. Since $F\subseteq [0,A] \cup \{x\in \mathbb R_{\geq 0}: |g(x)|\geq 1\}$ has finite measure 
and $g\in \Lspace{r}$, it follows that $g_2\in \Lspace{1}$ and, in particular, $\hat{g_2}$ is a bounded and measurable function. Therefore, $\hat{g_1}$  is not a locally integrable function. Since
$$\int_\rr \abs{g_1(x)}^{p(x)}dx = \int_E \abs{g(x)}^{p(x)}dx \leq \int_E \abs{g(x)}^{r}dx <+\infty,$$
this shows that $g_1\in L^{p(\cdot)}(\rr)$ but $\hat g_1$ is not a locally integrable function, which is the desired contradiction.
The case $\pinfml > 2$ is analogous taking $\supp (g) \subseteq \rr_{\leq 0}.$
\end{proof}

We now state one of the pillars of the proof of Theorem \ref{thm: necessary and sufficient thm}, a result that we believe is of independent interest. It relates the boundedness of the Fourier transform as an operator between variable Lebesgue spaces to its boundedness as an operator whose domain is given by the intersection of two classical Lebesgue spaces. Recall that, given two Banach spaces $X, Y$, we equip the space $X\cap Y$ with the norm
$$
\|f\|_{X\cap Y}:=\max\{\|f\|_X, \|f\|_Y\}.
$$
In what follows, given $h \in \rr$ and a function $f$, we denote by $\tau_h f$ the translation of $f$ given by $x \mapsto f(x-h)$. We also use the notation
$$
\langle f, g\rangle=\int_\rr f(x)g(x)\,dx
$$
whenever $f\in L^{p'(\cdot)}(\rr)$ and $g\in L^{p(\cdot)}(\rr)$, identifying $L^{p'(\cdot)}(\rr)$ with its canonical image in the dual space $(L^{p(\cdot)}(\rr))^*$. This identification is an isomorphism whenever $p(\cdot)\in \mathcal{P}_b(\rr)$ (see \cite[Theorem~2.80]{bo:cruz-uribe_fiorenza:variable_lebesgue_spaces}). Finally recall that, given a \emph{tempered distribution} $g\in \mathcal{S}'(\rr)$, its Fourier transform $\hat{g}\in \mathcal{S}'(\rr)$ is defined by
$$
\langle \hat{g}, f\rangle=\langle g, \hat{f}\rangle \qquad \text{for all $f\in \mathcal{S}(\rr)$,}
$$
where $\langle g, \hat{f}\rangle$ denotes the action of $g$ on $\hat{f}$ 
(see, for instance, \cite[Definition~2.3.7]{bo:grafakos:classical_fourier_analysis}). Since every $g\in L^{p(\cdot)}(\rr)$ defines a tempered distribution, we have
\begin{equation}\label{eq:duality}
\int_\rr g(x)\hat{f}(x)\,dx=\langle g, \hat f\rangle = \langle \hat{g}, f\rangle
\end{equation}
for every $g\in L^{p(\cdot)}(\rr)$ and every $f\in \mathcal{S}(\rr)$. Notice that the bracket notation $\langle \cdot, \cdot\rangle$ is used to denote both the integral and the distributional pairing. This should cause no confusion, since these notions agree whenever a distribution is represented by a function. Finally, we denote by $\overline{C^\infty_c(\rr)}^{L^{\pinfmr}\cap L^{\pinfMr}}$ the closure of the space $C^\infty_c(\rr)$ in $L^{\pinfmr}(\rr)\cap L^{\pinfMr}(\rr)$.

\begin{proposition}\label{prop: inclusion transformada de Fourier}
Let $p(\cdot), q(\cdot)\in \mathcal{P}(\rr)$ be such that the Fourier transform 
\[
\Fourier : L^{p(\cdot)}(\rr) \to L^{q(\cdot)}(\rr)
\]
is bounded. Then the operators
\[
\Fourier : \overline{C^\infty_c(\rr)}^{L^{\pinfmr}\cap L^{\pinfMr}}\longrightarrow L^{q(\cdot)}(\rr)
\quad \text{and} \quad 
\Fourier : \overline{C^\infty_c(\rr)}^{L^{\pinfml}\cap L^{\pinfMl}} \longrightarrow L^{q(\cdot)}(\rr)
\]
are bounded. Moreover, by duality,
\[
\Fourier : \overline{C^\infty_c(\rr)}^{\,L^{(q')_{\infty,r}^-}\cap L^{(q')_{\infty,r}^+}}
\longrightarrow L^{p'(\cdot)}(\rr)
\quad
\text{and}
\quad
\Fourier : \overline{C^\infty_c(\rr)}^{\,L^{(q')_{\infty,l}^-}\cap L^{(q')_{\infty,l}^+}}
\longrightarrow L^{p'(\cdot)}(\rr)
\]
are bounded.
\end{proposition}

\begin{proof} First let us note that, by Proposition \ref{prop: lim menor a 2}, the hypothesis implies $\pinfmr, \pinfml\leq 2$.

We prove that $\Fourier :  \overline{C^\infty_c(\rr)}^{L^{\pinfmr}\cap L^{\pinfMr}} \to L^{q(\cdot)}(\rr)$ is bounded, the other case is analogous. Let $0<\varepsilon<\pinfmr - 1$ and let $f\in C^\infty_c(\rr)$ be such that $\|f\|_{L^{\pinfmr - \epsilon}\cap L^{\pinfMr+\varepsilon}}=1$ (when $\pinfmr=1$, the argument can be adapted by replacing $\pinfmr-\varepsilon$ with $1$). By the hypothesis and the modulation property, there exists a constant $K>0$, independent of $\varepsilon$, such that
$$
\|\hat{f}\|_{L^{q(\cdot)}(\rr)}\leq K \|f\|_{L^{p(\cdot+h)}(\rr)}
$$ 
for all $h>0$. Choose $h_0>0$ sufficiently large so that 
$$\pinfmr - \epsilon \leq p(x) \leq \pinfMr + \epsilon\quad \text{for every $x\in \mathbb{R}$ such that $x-h_0\in \supp(f)$}$$ 
(note that $h_0$ depends on $f$) and define 
\[
	\tilde{p}(x) :=
	\begin{cases}
		p(x) &\text{if }  x-h_0\in \supp(f)\\
		\pinfmr &\text{if } \text{not}.
	\end{cases}
\]
Thus, $\rho_{p(\cdot + h_0)}(f)=\rho_{\tilde{p}(\cdot + h_0)}(f)$. Now, since $\pinfmr - \epsilon \leq \tilde{p}(\cdot)\leq \pinfMr+\varepsilon$, we have
$$
\rho_{\tilde{p}(\cdot)}(g)\leq \rho_{\pinfmr - \varepsilon}(g) + \rho_{\pinfMr+\varepsilon}(g) \quad \text{for every} \quad g\in L^{\pinfmr - \epsilon}(\rr)\cap L^{\pinfMr+\varepsilon}(\rr)$$
(see \cite[Eq.~(3.3.8)]{bo:diening_harjulehto_hasto_ruzicka:Lebesgue_and_Sobolev_spaces_with_variable_exponents}).
In particular, 
$$
\rho_{p(\cdot+h_0)}(f)=\rho_{\tilde{p}(\cdot+h_0)}(f)=\rho_{\tilde{p}(\cdot)}(\tau_{h_0}f)\leq \rho_{\pinfmr - \varepsilon}(\tau_{h_0}f) + \rho_{\pinfMr+\varepsilon}(\tau_{h_0}f)\leq 2
$$
and, consequently, $\|f\|_{p(\cdot+h_0)}\leq 2$. This proves that 
$\|\hat{f}\|_{q(\cdot)}\leq 2K$
for every $f\in C^\infty_c(\rr)$ satisfying $\|f\|_{L^{\pinfmr - \epsilon}\cap L^{\pinfMr+\varepsilon}}=1$. Therefore,
$$
\|\hat{f}\|_{L^{q(\cdot)}(\rr)}\leq K \|f\|_{L^{p(\cdot+h_0)}(\rr)}\leq 2 K \|f\|_{L^{\pinfmr - \epsilon}\cap L^{\pinfMr+\varepsilon}}
$$
for every $f\in C^\infty_c(\rr)$.
Taking limit $\varepsilon\to 0$ and applying the dominated convergence theorem we obtain
$$\|\hat{f}\|_{L^{q(\cdot)}(\rr)}\leq 2 K  \|f\|_{L^{\pinfmr}\cap L^{\pinfMr}}$$
for every $f\in C^\infty_c(\rr)$.
The first statement now follows by density.

For the dual statement, take any $f \in C^\infty_c(\rr)$. Then, by \eqref{eq:duality} and the boundedness of $\Fourier : L^{p(\cdot)}(\rr) \to L^{q(\cdot)}(\rr)$, for every $g \in L^{p(\cdot)}(\rr)$ with $\|g\|_{L^{p(\cdot)}(\rr)}=1$ we have
\[
|\langle g, \hat f\rangle| = |\langle \hat g, f\rangle|
\le C_1 \|f\|_{L^{q'(\cdot)}(\rr)} \|\hat g\|_{L^{q(\cdot)}(\rr)}
\le C_2 \|f\|_{L^{q'(\cdot)}(\rr)} \|g\|_{L^{p(\cdot)}(\rr)}
= C_2 \|f\|_{L^{q'(\cdot)}(\rr)}.
\]
Hence, by \cite[Proposition 2.79]{bo:cruz-uribe_fiorenza:variable_lebesgue_spaces}, there exists a constant $K>0$ such that
\[
\|\hat f\|_{L^{p'(\cdot)}(\rr)} \le K \|f\|_{L^{q'(\cdot)}(\rr)}.
\]
Observe that $K$ is independent of $f$. Arguing as in the first part, using the modulation property we obtain that for every $h>0$,
\[
\|\hat f\|_{L^{p'(\cdot)}(\rr)} \le K \|f\|_{L^{q'(\cdot+h)}(\rr)}.
\]
Given $\varepsilon>0$, proceeding as before, we can choose $h_0>0$ sufficiently large so that
\[
\|f\|_{L^{q'(\cdot+h_0)}(\rr)} 
\le 2 \|f\|_{L^{(q')_{\infty,r}^- -\varepsilon}(\rr)\cap L^{(q')_{\infty,r}^+ + \varepsilon}(\rr)},
\]
with the caveat that, if $(q')_{\infty,r}^- = +\infty$, the terms 
$(q')_{\infty,r}^- - \varepsilon$ and $(q')_{\infty,r}^+ + \varepsilon$ 
 are replaced by $\frac{1}{\varepsilon}$ and $+\infty$, respectively. 
Letting $\varepsilon \to 0$, we obtain
\[
\|\hat f\|_{L^{p'(\cdot)}(\rr)} 
\le 2K \|f\|_{L^{(q')_{\infty,r}^-}(\rr)\cap L^{(q')_{\infty,r}^+}(\rr)}.
\]
The result then follows by density. The other case is analogous.
\end{proof}

\begin{remark}\label{rmk: dual de la transf}
If we consider $p(\cdot), q(\cdot)\in \mathcal{P}_b(\rr)$ in Proposition~\ref{prop: inclusion transformada de Fourier}, then we have that
$$
\Fourier : L^{\pinfmr}(\rr)\cap L^{\pinfMr}(\rr)\to L^{q(\cdot)}(\rr)\quad \text{and}\quad \Fourier : \,L^{(q')_{\infty,r}^-}(\rr)\cap L^{(q')_{\infty,r}^+}(\rr)
\to L^{p'(\cdot)}(\rr)
$$
are bounded, and similarly for the essential limits at $-\infty$. Moreover, following the proof of the duality statement, it is not difficult to see that $$\Fourier\colon \overline{C^\infty_c(\rr)}^{L^{q'(\cdot)}}\to L^{p'(\cdot)}(\rr)$$ is bounded whenever $\Fourier : L^{p(\cdot)}(\rr) \to L^{q(\cdot)}(\rr)$ is.
Hence, if $p(\cdot), q(\cdot)\in \mathcal{P}_b(\rr)$, the operator $\Fourier : L^{p(\cdot)}(\rr) \to L^{q(\cdot)}(\rr)$ is bounded if and only if $\Fourier : L^{q'(\cdot)}(\rr) \to L^{p'(\cdot)}(\rr)$ is bounded. The essential point is that $C_c^\infty(\rr)$ is dense in $L^{p(\cdot)}(\rr)$ whenever $p_+<\infty$ (see \cite[Theorem~3.4.12]{bo:diening_harjulehto_hasto_ruzicka:Lebesgue_and_Sobolev_spaces_with_variable_exponents}).
\end{remark}

\subsection{Proof of Theorem \ref{thm: necessary and sufficient thm}.}
We are now in a position to prove Theorem~\ref{thm: necessary and sufficient thm}. We begin with the easy implication $(i)\Rightarrow (ii)$. In virtue of \cite[Theorem~3.3.1]{bo:diening_harjulehto_hasto_ruzicka:Lebesgue_and_Sobolev_spaces_with_variable_exponents} and the hypotheses in $(i)$, we have that
$$
L^{p(\cdot)}(\mathbb{R})\hookrightarrow L^{p_-}(\mathbb{R}) \quad \text{and}\quad L^{(p_-)'}(\mathbb{R})\hookrightarrow L^{q(\cdot)}(\mathbb{R})
$$
with continuous inclusions (the norm of these inclusions depends on $\|1\|_{r(\cdot)}$ and $\|1\|_{s(\cdot)}$). Taking this into account and applying the classical Hausdorff--Young inequality, we obtain
$$
\|\hat{f}\|_{q(\cdot)}\lesssim \|\hat{f}\|_{(p_-)'}\leq \|f\|_{p_-}\lesssim \|f\|_{p(\cdot)},
$$
which is the desired statement.

We now turn our attention to the implication $(ii) \Rightarrow (iii)$. We assume that the Fourier transform $\Fourier : L^{p(\cdot)}(\mathbb{R})\to L^{q(\cdot)}(\mathbb{R})$ is bounded and we need to prove that the conditions
$$
\pinfmr=\pinfml = p_-\leq 2, \quad \qinfMr=\qinfMl =q_+ \quad  \text{and} \quad (p_-)'=q_+
$$
hold. 

\noindent\textbf{Step 1.}
Let us begin by showing that $\pinfmr,\pinfml\leq (q')_-$. It suffices to prove the inequality $\pinfmr\leq (q')_-$, since the proof of $\pinfml\leq (q')_-$ is analogous. Suppose, to the contrary, that $\pinfmr> (q')_-$. Then, there exist $1\leq r'<\pinfmr$ and a bounded set of positive measure $E$ such that $$E\subseteq\{x\in \mathbb R: q'(x)\leq r'\}.$$ By Proposition~\ref{prop: inclusion transformada de Fourier}, it suffices to show that there exists a function $$f\in \overline{C^\infty_c(\rr)}^{L^{\pinfmr}\cap L^{\pinfMr}}$$ such that $\hat f\notin L^{q(\cdot)}(\mathbb R)$. By considering translations of $f$ and subsets of $E$, we may suppose that $E\subseteq [-1,1]$ and that $0$ is a Lebesgue density point of $E$.
Let $0<\alpha<1$ be such that $\frac{1}{r}<\alpha<1-\frac{1}{\pinfmr}$, where $r$ is the conjugate of $r'$, and let $f$ be the inverse Fourier transform of $g(x):=|x|^{-\alpha} \chi_{[-1,1]}(x)$.

On the one hand, since $g\in L^1(\rr)$, the function $f$ is continuous (moreover, $f\in C_0(\rr)$) and belongs to $L^\infty(\rr)$. We will show that $f\in  L^{{\pinfmr}}(\mathbb R)$. This,
together with the inclusion $L^{{\pinfmr}}(\mathbb R)\cap L^{\infty}(\mathbb R)\subseteq L^{\pinfmr}(\mathbb R)\cap L^{{\pinfMr}}(\mathbb R)$ (see \cite[Theorem~3.3.11]{bo:diening_harjulehto_hasto_ruzicka:Lebesgue_and_Sobolev_spaces_with_variable_exponents}), yields $$f\in  L^{{\pinfmr}}(\mathbb R)\cap L^{{\pinfMr}}(\mathbb R).$$ Since $f\in C_0(\rr)$, it follows that $$f\in \overline{C^\infty_c(\rr)}^{L^{\pinfmr}\cap L^{\pinfMr}}.$$
To prove that $f\in  L^{{\pinfmr}}(\mathbb R)$ it suffices to show that
\begin{equation}\label{eq: f en Lpinf}
\int_{|x|>1} |f(x)|^{\pinfmr}\, dx <\infty. 
\end{equation}
Since $g$ is a real and even function, then $f$ is real and even and we have
$$
f(x)= \int_\rr g(\xi)\cos(2\pi x\xi)\, d\xi = \int_{-1}^1 |\xi|^{-\alpha} \cos(2\pi x\xi)\, d\xi = 2\int_0^1 \xi^{-\alpha} \cos(2\pi x\xi)\, d\xi.
$$
Then, a simple change of variables yields
$$
f(x)=2 (2\pi x)^{\alpha-1}\int_0^{2\pi x}t^{-\alpha} \cos(t)\, dt \qquad \text{for $x>1$}
$$
and, since
$$
\int_0^\infty t^{s-1}\cos(t)=\Gamma(s) \cos\left(\frac{\pi s}{2}\right) \qquad \text{for $0<\text{Re}(s)<1$}
$$
(see, for instance, the table of Mellin transforms in \cite{gradshteyn2014table}), we obtain
\begin{eqnarray*}
f(x)=2 (2\pi x)^{\alpha-1} \left[\Gamma(1-\alpha) \sin\left(\frac{\pi \alpha}{2}\right) - \int_{2\pi x}^\infty t^{-\alpha} \cos(t)\, dt\right] \quad \text{for $x>1$.}
\end{eqnarray*}
An integration by parts argument shows that
$$
f(x)= 2 (2\pi x)^{\alpha-1} \Gamma(1-\alpha) \sin\left(\frac{\pi \alpha}{2}\right) + O(x^{-1}) \quad \text{for $x>1$}
$$
and, since $\pinfmr(1-\alpha)>1$ and $f$ is an even function, then \eqref{eq: f en Lpinf} holds.

On the other hand, let us show that $\hat{f}\notin \Lspace{q(\cdot)}$. Since
$f=\check{g}$, we have $\hat{f}=g$ in the sense of tempered distributions. As
$g\in L^1(\mathbb{R})$, it follows that $\hat{f}=g$ almost everywhere. Now, if $g$ were in $L^{q(\cdot)}(\mathbb R)$, then we would have $g\in L^{q(\cdot)}(E)$ and, since $E$ is bounded and $r\leq q(x)$ for every $x\in E$, it would follow that $g\in L^{r}(E)$. Let us prove that this is not the case.
Noting that $|g(x)|^r\geq |x|^{-1}$ for every $x\in E$ (recall that $r\alpha>1$), it suffices to prove that $|x|^{-1}\chi_E\notin L^1(E)$. This will follow from the fact that $0$ is a Lebesgue density point of $E$, that is, $$\frac{|E\cap [-\varepsilon,\varepsilon]|}{2\varepsilon} \xrightarrow[\epsilon \to 0]{} 1.$$ 
Indeed, there exists $\varepsilon_0>0$ such that $|E\cap [-\varepsilon,\varepsilon]|>\varepsilon$ for every $0<\varepsilon\leq \varepsilon_0$. Then, taking
$$
A_n=\left(-\frac{\varepsilon_0}{4^n}, -\frac{\varepsilon_0}{4^{n+1}}\right] \cup \left[ \frac{\varepsilon_0}{4^{n+1}}, \frac{\varepsilon_0}{4^{n}}\right)
$$
and noting that the sets $A_n$ are pairwise disjoint, we have
\begin{equation}\label{eq: not in L1}
\int_E |x|^{-1} dx \geq \sum_{n=1}^\infty\int_{E\cap A_n}|x|^{-1}dx\geq \sum_{n=1}^\infty \frac{4^n}{\varepsilon_0}\cdot |E\cap A_n|. 
\end{equation}
But, since 
$$
|E\cap A_n|=\left|E\cap [-\varepsilon_0/4^n, \varepsilon_0/4^n]\right| - |E\cap[-\varepsilon_0/4^{n+1}, \varepsilon_0/4^{n+1}]| > \frac{\varepsilon_0}{4^{n}} - \frac{2\varepsilon_0}{4^{n+1}}=\frac{\varepsilon_0}{4^n2},
$$
it follows from \eqref{eq: not in L1} that $|x|^{-1}\chi_E\notin L^1(E)$.

\noindent\textbf{Step 2.} Now we show that $(q')_-\leq p_-$. Since $(q')_- \le (q')_{\infty,r}^-$, it suffices to show that $(q')_{\infty,r}^- \le p_-$. At this point, it is worth noting that if the exponents $p(\cdot)$ and $q(\cdot)$ were bounded, then 
$
\Fourier : L^{q'(\cdot)}(\rr) \to L^{p'(\cdot)}(\rr)
$ would be bounded by Remark~\ref{rmk: dual de la transf}, and the statement would follow from Step~1. In the general case, however, we argue again by contradiction and assume that $p_- < (q')_{\infty,r}^-$.
Proceeding as in Step~1, \emph{mutatis mutandis}, we consider a bounded set $E$ of positive measure such that $$E \subseteq \{x \in \rr : p(x) \le r'\},$$ where $1\leq r'<(q')_{\infty,r}^-$, and define a function $g\in L^1(\rr)$ such that
$$
\check{g}\in L^{(q')_{\infty,r}^-}(\rr)\cap L^{\infty}(\rr) \subseteq L^{(q')_{\infty,r}^-}(\rr)\cap L^{(q')_{\infty,r}^+}(\rr)\quad \text{and}\quad g\notin L^{p'(\cdot)}(\rr).
$$
Now, since $f:=\check{g}\in C_0(\rr)$, we have $$f\in \overline{C^\infty_c(\rr)}^{\,L^{(q')_{\infty,r}^-}\cap L^{(q')_{\infty,r}^+}}$$ and $\hat{f}=g\notin L^{p'(\cdot)}(\rr)$. This contradicts Proposition~\ref{prop: inclusion transformada de Fourier}.

\noindent\textbf{Step 3.} From the previous steps, we deduce that 
$$
\pinfmr=\pinfml = p_-\leq 2 \qquad \text{and}\qquad (p_-)'=q_+.
$$
Indeed, since we always have $p_-\leq \pinfmr$ and $p_-\leq \pinfml$, in virtue of Steps~1 and 2, we have
$$
p_-\leq \pinfmr, \pinfml\leq (q')_-\leq p_-
$$
and, consequently, $\pinfmr=\pinfml = p_-$ and $(p_-)'=((q')_-)'=q_+$. The fact that $p_-\leq 2$ follows now from Proposition~\ref{prop: lim menor a 2}.

\noindent\textbf{Step 4.} Finally, let us show that $\qinfMr=\qinfMl =q_+$. A careful look at the proof of Step~2 shows that
$$
(\qinfMr)'=(q')_{\infty, r}^-\leq p_- \qquad \text{and}\qquad (\qinfMl)'=(q')_{\infty, l}^-\leq p_-.
$$
Since $p_-=(q_+)'$ by Step~3, and since we always have $\qinfMr, \qinfMl \leq q_+$, we deduce the desired statement.
\hfill\qedsymbol

As a first consequence of Theorem~\ref{thm: necessary and sufficient thm}, taking $q(\cdot) := p'(\cdot)$ (which implies $r(\cdot) = s(\cdot)$), we obtain sufficient conditions for the validity of the \emph{classical} Hausdorff--Young inequality with conjugate exponents. These conditions, which concern the regularity of the exponent at infinity, are stronger than those in item $(i)$ of Theorem~\ref{thm: necessary and sufficient thm} but are somewhat easier to verify.
Recall that a function $r(\cdot)$ is \emph{log-H\"older continuous at infinity}, denoted $r(\cdot) \in \LH_\infty$, if there exist constants $C_\infty$ and $r_\infty$ such that for all $x \in \Omega$
\[
\abs{r(x) - r_\infty} \le \frac{C_\infty}{\log(e + \abs{x})}.
\]
Combining Theorem~\ref{thm: necessary and sufficient thm} with
\cite[Remark~2.46]{bo:cruz-uribe_fiorenza:variable_lebesgue_spaces},
we obtain the following corollary.

\begin{corollary}\label{cor:p- and LHinfty imply F bounded}
Let $p(\cdot)\in \mathcal{P}(\rr)$ be such that $1/p(\cdot) \in \LH_\infty$ and
$
p_-=\lim_{|x|\to +\infty}p(x)\leq  2.
$
Then, the Fourier transform $\Fourier\colon L^{p(\cdot)}(\mathbb R)\to L^{p'(\cdot)}(\mathbb R)$ is bounded.
\end{corollary}

Another consequence of Theorem~\ref{thm: necessary and sufficient thm} is the following corollary, which provides necessary conditions for the boundedness of $\Fourier :L^{p(\cdot)}(\mathbb R)\to L^{q(\cdot)}(\mathbb R)$ when either $p(\cdot)$ or $q(\cdot)$ admits a limit at infinity. This result will play a key role in Section~\ref{Seccion caracteristicas}. It is worth noting that a related result, in the setting where the exponent $p(\cdot)$ is continuous and non-decreasing on $(0,+\infty)$, was established by Saucedo and Tikhonov in the proof of \cite[Theorem~1.1]{ar:saucedo_tikhonov:2025:note_on_fourier_inequalities}. 

\begin{corollary}\label{cor:Fourier bounded from p_inf to p var}
Let $p(\cdot), q(\cdot) \in \mathcal{P}(\rr)$ be such that $\Fourier :L^{p(\cdot)}(\mathbb R)\to L^{q(\cdot)}(\mathbb R)$ is bounded. Then, the following assertions hold.
\begin{enumerate}
\item If $\ds \lim_{|x|\to +\infty} p(x)=p_\infty$, then $p_\infty\leq 2$ and $\Fourier :L^{p_\infty}(\mathbb R)\to L^{q(\cdot)}(\mathbb R)$ is bounded.
\item If $\ds \lim_{|x|\to +\infty} q(x)=q_\infty$, then $q_\infty\geq 2$ and 
$$
\Fourier :L^{q_\infty'}(\mathbb R)\to L^{p'(\cdot)}(\mathbb R)\quad \text{and} \quad \Fourier :L^{p(\cdot)}(\mathbb R)\to L^{q_\infty}(\mathbb R)
$$
are bounded.
\end{enumerate}
\end{corollary}

\begin{proof}
The inequality $p_\infty\leq 2$ in item $(i)$ follows from implication $(ii)\Rightarrow (iii)$ in Theorem~\ref{thm: necessary and sufficient thm}. The boundedness of the Fourier transform is an immediate consequence of Proposition~\ref{prop: inclusion transformada de Fourier} (see also Remark~\ref{rmk: dual de la transf}), since
$$
p_\infty=\pinfmr=\pinfMr <\infty.
$$
 The same implication $(ii)\Rightarrow (iii)$ in Theorem~\ref{thm: necessary and sufficient thm} shows that $q_\infty\geq 2$ in item $(ii)$, while the boundedness of $\Fourier :L^{q_\infty'}(\mathbb R)\to L^{p'(\cdot)}(\mathbb R)$ follows from the second part of Proposition~\ref{prop: inclusion transformada de Fourier}.
Finally, let us prove that $\Fourier :L^{p(\cdot)}(\mathbb R)\to L^{q_\infty}(\mathbb R)$ is bounded. By hypothesis, given $f\in L^{p(\cdot)}(\mathbb R)$ we know that $\hat f$ is a measurable function. By \cite[Proposition~2.79]{bo:cruz-uribe_fiorenza:variable_lebesgue_spaces} it suffices to prove that there exists a constant $K>0$ such that
$$|\langle \hat f, g\rangle| \leq K \|f\|_{L^{p(\cdot)}(\rr)}\| g\|_{L^{q_\infty'}(\mathbb R)}$$
for all $g\in L^{q'_\infty}(\rr)$. Since $q'_\infty\leq 2$, it is enough to verify this for $g\in \Schwartz(\rr)$. Now, by \eqref{eq:duality} and the boundedness of $\Fourier :L^{q_\infty'}(\mathbb R)\to L^{p'(\cdot)}(\mathbb R)$ we have
\begin{eqnarray*}
    |\langle \hat f, g\rangle| =|\langle  f, \hat g\rangle| \leq C_1 \| f\|_{L^{p(\cdot)}(\rr)} \|\hat g\|_{L^{p'(\cdot)}(\rr)}
    \leq  C_2\| f\|_{L^{p(\cdot)}(\mathbb R)} \| g\|_{L^{q_\infty'}(\mathbb R) }
\end{eqnarray*}
for all $g\in \Schwartz(\rr)$. This completes the proof.
\end{proof}

It is worth noting that the conclusions of the previous corollary remain valid if we only assume that the exponents $p(\cdot)$ and $q(\cdot)$ admit a one-sided limit at infinity (either at $+\infty$ or at $-\infty$). The proof is unchanged. We will return to this observation in Remark~\ref{rem: one sided limits}.

\subsection{Optimal relations between the exponents}

We dedicate this subsection to showing that, if the Fourier transform $\Fourier : \Lspace{p(\cdot)} \to \Lspace{q(\cdot)}$ is bounded, then the relations between the exponents found in item $(iii)$ of Theorem \ref{thm: necessary and sufficient thm} are, in some sense, optimal. That is, there are no inequalities relating 
$p_+$, $\pinfMl$, $\pinfMr$, $q_-$, $\qinfml$ or $\qinfmr$, other than the trivial ones. The underlying reason for this is that the boundedness of the Fourier transform is stable under some perturbations of the exponents on sets of finite measure (see Lemma~\ref{lemma:modificar exponentes en medida finita} below).

The relations between the exponents can be summarized in the following diagram. We write $a \lra b$ to indicate that $a \le b$. A dashed arrow between two nodes indicates that neither of the corresponding inequalities holds in general. We write $\qinfmstar$ (respectively, $\pinfMstar$) to denote both $\qinfml$ and $\qinfmr$ (respectively, $\pinfMl$ and $\pinfMr$).
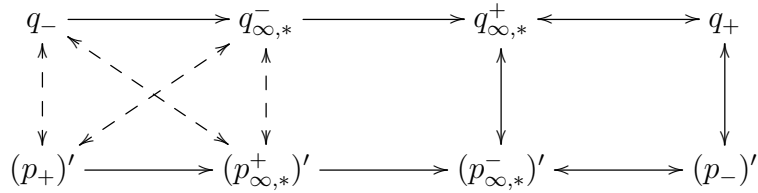
\[
\xymatrixrowsep{3pc}\xymatrixcolsep{4pc}\xymatrix{
	q_- \ar[r] \ar@{<-->}[d] \ar@{<-->}[dr] & \qinfmstar \ar[r] \ar@{<-->}[d] \ar@{<-->}[dl] & \qinfMstar \ar@{<->}[r] \ar@{<->}[d] & q_+ \ar@{<->}[d] \\
	(p_+)' \ar[r] & (\pinfMstar)' \ar[r] & (\pinfmstar)' \ar@{<->}[r] & (p_-)'
}
\]
All that remains is to prove the dashed arrows, since the other ones are either trivial or included in Theorem \ref{thm: necessary and sufficient thm}. To this end, we will use Corollary \ref{cor:p- and LHinfty imply F bounded}, together with the following lemma, which is a simple consequence of the embeddings between variable Lebesgue spaces.

\begin{lemma}\label{lemma:modificar exponentes en medida finita}
Let $p(\cdot), q(\cdot)\in \mathcal{P}(\mathbb{R})$ and let $T\colon L^{p(\cdot)}(\mathbb{R})\to L^{q(\cdot)}(\mathbb{R})$ be a bounded linear operator. Suppose $\Omega\subset \mathbb{R}$ is a measurable set with $|\Omega|<\infty$ and let $r(\cdot),s(\cdot)\in \mathcal{P}(\mathbb{R})$ be such that $r(x)\geq p(x)$, $s(x) \le q(x)$ for every $x\in \Omega$ and $r(x)=p(x)$, $s(x) = q(x)$ for every $x\in \mathbb{R}\setminus \Omega$. Then, $T\colon L^{r(\cdot)}(\mathbb{R})\to L^{s(\cdot)}(\mathbb{R})$ is a bounded operator.
\end{lemma}
\begin{proof}
It follows by noticing that, under these hypotheses, we have the embeddings $\Lspace{r(\cdot)} \hookrightarrow \Lspace{p(\cdot)}$ and $\Lspace{q(\cdot)} \hookrightarrow \Lspace{s(\cdot)}$ (see \cite[Theorem~3.3.1]{bo:diening_harjulehto_hasto_ruzicka:Lebesgue_and_Sobolev_spaces_with_variable_exponents}).
\end{proof}

We now present the examples corresponding to the dashed arrows in the previous diagram.  

\begin{enumerate}
    \item Let $I_n := (n,n+2^{-n})$ and define
	\[
	\Omega := \bigcup_{n=1}^\infty I_n.
	\]
	Notice that $\abs{\Omega} < \infty$ and consider
	\[
	r(x) := \left( 1 - \frac{1/2}{\ln(e + \abs{x})} \right)^{-1}.
	\]
	Then $r_- = \lim_{\abs{x} \to +\infty} r(x)=1$ and $1/r(\cdot) \in \LH_\infty$. Hence, by Corollary \ref{cor:p- and LHinfty imply F bounded}, we know that $\Fourier : L^{r(\cdot)}(\rr) \to L^{r'(\cdot)}(\rr)$ is bounded.
	Now, by Lemma~\ref{lemma:modificar exponentes en medida finita}, letting
	\begin{equation*}
	p(x) :=
	\begin{cases}
		2 &\text{if } x \in \Omega \\
		r(x) &\text{if } x \not\in \Omega
	\end{cases}
    \quad 
    \text{and}
    \quad
    q(x) :=
	\begin{cases}
		\frac{3}{2} &\text{if } x \in \Omega \\
		r'(x) &\text{if } x \not\in \Omega,
	\end{cases}
	\end{equation*}
	we have that $\Fourier : L^{p(\cdot)}(\rr) \lra L^{q(\cdot)}(\rr)$ is bounded. Since
   $$
   q_-=\qinfmr=\frac{3}{2}\qquad \text{and}\qquad (p_+)'=(\pinfMr)'=2,
   $$  
this is an example for which 
$$
\qinfmr < (\pinfMr)', \quad q_-<(\pinfMr)', \quad \qinfmr<(p_+)'\quad \text{and}\quad q_-<(p_+)'.
$$
Analogously, we can construct examples for which the same inequalities hold, with either (or both) $\qinfmr$ and $(\pinfMr)'$ replaced by $\qinfml$ and $(\pinfMl)'$. 

\medskip

\item Take $\Omega$, $r(\cdot)$ and $q(\cdot)$ as in (i), and
    \[
	p(x) :=
	\begin{cases}
		4 &\text{if } x \in \Omega \\
		r(x) &\text{if } x \not\in \Omega.
	\end{cases}
	\]
    Then $\Fourier : L^{p(\cdot)}(\rr) \to L^{q(\cdot)}(\rr)$ is bounded in virtue of Lemma~\ref{lemma:modificar exponentes en medida finita}. Since
    $$
    q_-=\qinfmr=\frac{3}{2}\qquad \text{and}\qquad (p_+)'=(\pinfMr)'=\frac{4}{3},
    $$
    this is an example for which
    $$
    (\pinfMr)'<\qinfmr, \quad (\pinfMr)'<q_-, \quad (p_+)'<\qinfmr\quad \text{and}\quad (p_+)'<q_-.
    $$
As mentioned in (i), we can construct examples for which the same inequalities hold, with either (or both) $\qinfmr$ and $(\pinfMr)'$ replaced by $\qinfml$ and $(\pinfMl)'$.
\end{enumerate}

\section{The case $p(x)\to 1$ and $q(x)\to \infty$}\label{Seccion caracteristicas}

In view of Theorem \ref{thm: necessary and sufficient thm}, it is natural to ask whether any of its converses might hold. In this short section, we show that if $p(\cdot), q(\cdot)\in \mathcal{P}(\rr)$ satisfy
$$
\lim_{|x|\to +\infty} p(x)=1\qquad \text{and}\qquad \lim_{|x|\to +\infty}q(x)=+\infty,
$$
then conditions $(i)$ and $(ii)$ of Theorem~\ref{thm: necessary and sufficient thm} are equivalent. To prove this result, stated in Theorem~\ref{Thm: equivalencia q_+=infinito}, we follow ideas of Benedetto and Heinig in \cite[Theorem 2]{ar:benedetto_heinig:1983:weighted_Hardy_spaces_and_the_Laplace_transform}, concerning the behavior of $\|\chi_{[-t,t]}\|_{L^{q(\cdot)}(\rr)}$.

\begin{lemma}\label{lemma: condicion necesaria caracteristicas version limite}
Let $p(\cdot), q(\cdot)\in \mathcal{P}(\rr)$ be such that $\Fourier : L^{p(\cdot)}(\rr)\to L^{q(\cdot)}(\rr)$ is bounded and
$
\lim_{|x| \to +\infty} p(x) = 1.
$
Then,
\[
\sup_{t>0}\|\chi_{[-t,t]}\|_{L^{q(\cdot)}(\rr)} < +\infty.
\]
\end{lemma}

\begin{proof}
Let us consider $f:=\chi_{[-t,t]}$. On the one hand, by the hypotheses and Corollary~\ref{cor:Fourier bounded from p_inf to p var} we know that 
\begin{equation}\label{eq: norma de la transformada de la caracteristica}
    \|\hat f\|_{L^{q(\cdot)}(\rr)}\lesssim \|\chi_{[-t,t]}\|_{L^1(\rr)}=2t.
\end{equation} 
On the other hand, we have
\[
\hat{f}(\xi) = \int_\rr f(x) e^{-2\pi i x\xi} \ dx = 2 \int_{0}^t \cos(2\pi x\xi) \ dx.
\]
Thus, for any fixed $t>0$ we obtain 
\[
\|\hat f\|_{L^{q(\cdot)}(\rr)}\geq
2 \left\| \chi_{[-\frac{1}{2\pi t},\frac{1}{2\pi t}]}(\cdot) \int_0^t \cos(2\pi x\,\cdot) \ dx \right\|_{\Lspace{q(\cdot)}} \ge t \norm{\chi_{[-\frac{1}{2\pi t},\frac{1}{2\pi t}]}}_{\Lspace{q(\cdot)}},
\]
since $\cos(2\pi x\xi) \ge \frac{1}{2}$ whenever $\xi\in \left[-\frac{1}{2\pi t}, \frac{1}{2\pi t}\right]$ and $x\in [0,t]$.
Combining this estimation with \eqref{eq: norma de la transformada de la caracteristica} we deduce that 
$$
\sup_{t>0} \|\chi_{[-t,t]}\|_{L^{q(\cdot)}(\rr)}<\infty,
$$
which is the desired statement.
\end{proof}

\begin{proof}[Proof of Theorem~\ref{Thm: equivalencia q_+=infinito}]
If $1\in L^{p'(\cdot)}(\rr)\cap L^{q(\cdot)}(\rr)$, then the boundedness of $\Fourier\colon L^{p(\cdot)}(\rr)\to L^{q(\cdot)}(\rr)$ follows from the implication $(i)\Rightarrow (ii)$ in Theorem~\ref{thm: necessary and sufficient thm}. Just note that, under the hypotheses on the exponents $p(\cdot)$ and $q(\cdot)$, we have $r(\cdot)=p'(\cdot)$ and $s(\cdot)=q(\cdot)$.
Conversely, assume that $\Fourier\colon L^{p(\cdot)}(\rr)\to L^{q(\cdot)}(\rr)$ is bounded.
On the one hand, by Lemma~\ref{lemma: condicion necesaria caracteristicas version limite} we know that
\[
\sup_{n\in \mathbb{N}}\|\chi_{[-n,n]}\|_{L^{q(\cdot)}(\rr)}<\infty.
\]
Then, since $(\chi_{[-n,n]})_{n\in \mathbb{N}}$ increases pointwise to 1, we deduce from \cite[Theorem 2.59]{bo:cruz-uribe_fiorenza:variable_lebesgue_spaces} that $1\in L^{q(\cdot)}(\rr)$. On the other hand, applying Corollary~\ref{cor:Fourier bounded from p_inf to p var}, we obtain that  
$\Fourier:L^{1}(\rr)\to L^{p'(\cdot)}(\rr)$ is bounded. Hence, Lemma~\ref{lemma: condicion necesaria caracteristicas version limite} yields
\[
\sup_{n\in \mathbb{N}}\|\chi_{[-n,n]}\|_{L^{p'(\cdot)}(\rr)}<\infty,
\]
and the previous argument, applied to $p'(\cdot)$ instead of $q(\cdot)$, shows that $1 \in \Lspace{p'(\cdot)}$.
\end{proof}

\begin{remark}\label{rem: one sided limits}
   A closer inspection of the proofs of Corollary~\ref{cor:Fourier bounded from p_inf to p var} and Lemma~\ref{lemma: condicion necesaria caracteristicas version limite} reveals that it suffices to assume that the exponents $p(\cdot)$ and $q(\cdot)$ admit either a limit at $+\infty$ or a limit at $-\infty$ to obtain the same conclusions. Consequently, the equivalence in Theorem~\ref{Thm: equivalencia q_+=infinito} still holds under the weaker assumption that the exponents admit a one-sided limit.
\end{remark}

\section{On the sets where $p(x)>2$ and $q(x)<2$}\label{Seccion lagunares}

In the classical Hausdorff--Young inequality, the condition $p\le 2$ is essential for the boundedness of the Fourier transform. In the variable exponent setting, however, the exponent $p(\cdot)$ may exceed $2$ on certain subsets of $\mathbb{R}$. It is therefore natural to ask how large such sets can be if the Fourier transform is bounded from $L^{p(\cdot)}(\mathbb{R})$ into $L^{q(\cdot)}(\mathbb{R})$. 

We already have a first answer to this question. In view of Corollary \ref{cor:p- and LHinfty imply F bounded}, it is possible to have $p(x) > 2$ for all $x\in \rr$ provided that $p(\cdot)$ converges sufficiently fast to $p_- = 2$. Thus, the next natural question is how large the set $\{x\in \rr : p(x) > 2+\varepsilon\}$ can be. Correspondingly, the dual question is how large the set $\{x\in \rr : q(x) < 2-\varepsilon\}$ can be.
In this section, we obtain restrictions on the structure of these sets. More precisely, we show that if the set $\{x \in \rr: p(x) > 2+\varepsilon\}$ is contained, up to a set of finite measure, in a $\sigma$-elementary set (i.e., a countable disjoint union of intervals), then the lengths of the corresponding intervals must necessarily converge to zero. A similar result holds for $\{x\in \rr : q(x) < 2-\varepsilon\}$.

In the following, we make use of lacunary sequences. Recall that a sequence of positive integers $\Lambda=\{\lambda_j\}_{j=1}^\infty$ is called \emph{lacunary} if there exists a constant $A>1$ such that $\lambda_{j+1}\ge A\lambda_j$ for every $j\in \mathbb{N}$. Fourier series whose frequencies belong to such sequences enjoy remarkable properties; in particular, their $L^p$ norms are all comparable. More precisely, if $f$ is a function on the one-dimensional torus $\mathbb{T}$ (i.e., $[0,1]$ with the endpoints identified) whose Fourier coefficients vanish outside a lacunary set $\Lambda$, then for every $1\le p<\infty$ the norms $\|f\|_{L^p[0,1]}$ are equivalent up to constants depending only on $p$ and the lacunarity constant $A$. In what follows we will make use of this fact, whose precise statement can be found in \cite[Theorem~3.6.4]{bo:grafakos:classical_fourier_analysis}.

\begin{theorem}\label{thm:lagunares}
Let $\varepsilon>0$ and let $p(\cdot), q(\cdot)\in \mathcal{P}(\rr)$ be such that $\Fourier:L^{p(\cdot)}(\mathbb{R}) \to L^{q(\cdot)}(\mathbb{R})$ is bounded. 
\begin{enumerate}
    \item If $G$ is the disjoint union of a sequence of intervals $\{I_j\}_{j\in \N}$, 
    \[
 \{x\in\rr : p(x) > 2 + \varepsilon\}\subseteq G\quad \text{ and }\quad |G \setminus \{x\in\rr : p(x) > 2 + \varepsilon\}| < +\infty,
\]
then $|I_j|\rightarrow 0$ as $j \to +\infty$.
\item If $U$ is the disjoint union of a sequence of intervals $\{J_j\}_{j\in \N}$,
\[
 \{x\in\rr : q(x) < 2 - \varepsilon\}\subseteq U \quad \text{ and }\quad|U \setminus \{x\in\rr : q(x) < 2 - \varepsilon\}| < +\infty,
\]
then $|J_j|\rightarrow 0$ as $j \to +\infty$.
\end{enumerate}
\end{theorem}

\begin{proof}

	Let us prove $(i)$ by contradiction. Suppose that there exists a set $G$, which is the disjoint union of intervals $\{I_j\}_{j\in \N}$, such that 
    \[
 \{x\in\rr : p(x) > 2 + \varepsilon\}\subseteq G,\quad |G \setminus \{x\in\rr : p(x) > 2 + \varepsilon\}| < +\infty \quad \text{ and }\quad |I_{j}|\not\to 0.
\]
Then, there exist $r > 0$, a sequence $\{\lambda_j\}_j$ of positive integers and a non-zero integer $M$ such that
	\[
	E:=\bigcup_{j=1}^\infty\left(\frac{\lambda_j}{M} - r, \frac{\lambda_j}{M} + r\right) \subseteq G,
	\]
    where the intervals are pairwise disjoint.
    By selecting a suitable subcollection of the intervals, we may assume that $\{\lambda_j\}_j$ form a lacunary sequence, which will be useful later. Since $G$ coincides with $\{x\in\rr : p(x) > 2 + \varepsilon\}$ up to a set of finite measure, by modifying $p(\cdot)$ accordingly (using Lemma \ref{lemma:modificar exponentes en medida finita}), we may further suppose that
	$$E \subseteq \{x\in\rr : p(x) > 2 + \varepsilon\}.$$ 
Now, assuming \( M>0 \) (the case \( M<0 \) is entirely analogous), we have the following chain of continuous operators:
\begin{equation}\label{eq: chain}
L^{p(\cdot)}(E) \hookrightarrow L^{p(\cdot)}(\mathbb{R})
\xrightarrow{\Fourier} L^{q(\cdot)}(\mathbb{R})
\xrightarrow{\text{restr.}} L^{q(\cdot)}[0,M]
\hookrightarrow L^{1}[0,M].
\end{equation}
Our goal is to construct, for any sufficiently large $m\in \mathbb{N}$, 
a function $f_m$ satisfying
	\begin{equation}\label{eq: ineq fm}
	C\, m^{1/2} \le \|\hat{f}_m\|_{L^1[0,M]}
	\qquad \text{and} \qquad
	\|f_m\|_{L^{p(\cdot)}(E)} \le D\, m^{1/(2+\varepsilon)},
	\end{equation}
	where $C, D > 0$ are constants independent of $m$. 
	This, together with \eqref{eq: chain}, will yield the desired contradiction.

    Let $\varphi$ be a continuous function such that 
	$\operatorname{supp} \varphi \subseteq (-r, r)$, $\widehat{\varphi}$ is continuous and 
	$\widehat{\varphi}(x) > 0$ for all $x \in \mathbb{R}$\footnote{Such a function can be defined using
		\[
		\phi(x) := \chi_{[-1,1]} * \chi_{[-1,1]}(x) + e^{2\pi i \frac{x}{4}} \, \chi_{[-1,1]} * \chi_{[-1,1]}(x).
		\]
		
		Indeed, 
		\[
		\hat{\phi}(x) = \left( \frac{\sin(2\pi x)}{\pi x} \right)^2 + \left( \frac{\sin\big(2\pi(x+\frac{1}{4})\big)}{\pi (x+\frac{1}{4})} \right)^2= \left( \frac{\sin(2\pi x)}{\pi x} \right)^2 + \left( \frac{\cos(2\pi x)}{\pi (x+\frac{1}{4})} \right)^2. 
		\]}. For each  $m \in \mathbb{N}$ define
	\[
	f_m(x) = \sum_{j=1}^m \varphi\left(x - \frac{\lambda_j}{M}\right)\qquad \text{for $x\in E$}.
	\]
    Then, on the one hand, we have
	$$
	\hat{f}_m(\xi) = \sum_{j=1}^m e^{-2\pi i \frac{\lambda_j}{M} \xi}\, \widehat{\varphi}(\xi)
	$$
	and, consequently,
	\[
	\|\hat{f}_m\|_{L^1[0, M]} \ge 
	\min \{|\widehat{\varphi}(\xi)|: {0 \le \xi \le M} \} \,
	\bigg\| \sum_{j=1}^m e^{-2\pi i \frac{\lambda_j}{M} (\cdot)} \bigg\|_{L^1[0,M]}.
	\] 
Since $\{\lambda_j\}_j$ is a lacunary sequence, it follows from 
\cite[Theorem 3.6.4]{bo:grafakos:classical_fourier_analysis}
that there exists a constant $C>0$ such that
\[
\left\| \sum_{j=1}^m e^{2\pi i \frac{\lambda_j}{M} (\cdot)} \right\|_{L^1[0,M]}= |M| \left\| \sum_{j=1}^m e^{2\pi i \lambda_j (\cdot)} \right\|_{L^1[0,1]}
\ge C \left\| \sum_{j=1}^m e^{2\pi i \lambda_j (\cdot)} \right\|_{L^2[0,1]} = C m^{1/2}.
\]
This gives the first inequality in \eqref{eq: ineq fm}.

On the other hand, let us show that
\begin{equation}\label{eq: modular fm}
\rho_{p(\cdot)}\!\left(\frac{f_m}{(2m)^{\frac{1}{2+\varepsilon}}\|\varphi\|_{L^{2+\varepsilon}(E)}}\right) < 1
\end{equation}
for sufficiently large $m\in \mathbb{N}$, which yields the second inequality in \eqref{eq: ineq fm}.
Since we are interested in sufficiently large $m$, we may assume that
\begin{equation}\label{eq: lagunares m grande}
(2m)^{\frac{1}{2+\varepsilon}}\|\varphi\|_{L^{2+\varepsilon}(E)} > 2 \|\varphi\|_{L^\infty(E)}.
\end{equation}
Now,
\[
\rho_{p(\cdot)}\!\left(\frac{f_m}{(2m)^{\frac{1}{2+\varepsilon}}\|\varphi\|_{L^{2+\varepsilon}(E)}}\right)
=
\int_{E\setminus E_\infty}
\left|
\sum_{j=1}^m
\frac{\varphi\left(x-\frac{\lambda_j}{M}\right)}
{(2m)^{\frac{1}{2+\varepsilon}}\|\varphi\|_{L^{2+\varepsilon}(E)}}
\right|^{p(x)}
\,dx
+
\frac{\|f_m\|_{L^\infty(E_\infty)}}
{(2m)^{\frac{1}{2+\varepsilon}}\|\varphi\|_{L^{2+\varepsilon}(E)}},
\]
where $E_\infty := \{x\in E:\,\, p(x)=+\infty\}$. Noting that $\|f_m\|_{L^\infty(E_\infty)}=\|\varphi\|_{L^{\infty}(E)}$ (since $E$ is a union of pairwise disjoint intervals) and applying \eqref{eq: lagunares m grande}, we obtain the estimate
\[
\frac{\|f_m\|_{L^\infty(E_\infty)}}
{(2m)^{\frac{1}{2+\varepsilon}}\|\varphi\|_{L^{2+\varepsilon}(E)}}
=\frac{\|\varphi\|_{L^\infty(E)}}{{(2m)^{\frac{1}{2+\varepsilon}}\|\varphi\|_{L^{2+\varepsilon}(E)}}}<\frac{1}{2}.
\]
Recalling that $E \subset \{x\in\rr : p(x) > 2 + \varepsilon\}$ and using again \eqref{eq: lagunares m grande}, we obtain that
\begin{eqnarray*}
\int_{E\setminus E_\infty}
\left|
\sum_{j=1}^m
\frac{\varphi\left(x-\frac{\lambda_j}{M}\right)}
{(2m)^{\frac{1}{2+\varepsilon}}\|\varphi\|_{L^{2+\varepsilon}(E)}}
\right|^{p(x)}
\,dx
&\le&
\sum_{j=1}^m\int_{\frac{\lambda_j}{M}-r}^{\frac{\lambda_j}{M}+r}
\left|
\frac{\varphi\left(x-\frac{\lambda_j}{M}\right)}
{(2m)^{\frac{1}{2+\varepsilon}}\|\varphi\|_{L^{2+\varepsilon}(E)}}
\right|^{2+\varepsilon}
\,dx
\\
&=&
\frac{1}{2m\|\varphi\|_{L^{2+\varepsilon}(E)}^{2+\varepsilon}}
\sum_{j=1}^m
\int_{-r}^{r}
|\varphi\left(y\right)|^{2+\varepsilon}\,dy
\\
&=&\frac{1}{2}.
\end{eqnarray*}
Hence, inequality \eqref{eq: modular fm} holds, which completes the proof of $(i)$.

Item $(ii)$ also follows by contradiction. We only sketch the proof, since it follows along the same lines as that of $(i)$. If we assume that $|J_j|$ does not tend to zero, then, proceeding as before, eventually modifying $q(\cdot)$ on a set of finite measure, we obtain 
\[
E=\bigcup_{j=1}^\infty\left(\frac{\lambda_j}{M} - r, \frac{\lambda_j}{M} + r\right) \subseteq \{x\in\rr : q(x) < 2 - \varepsilon\}=\{x\in\rr : q'(x) > 2 + \tilde \varepsilon\},
\]
where $M$ is a non-zero integer, $\{\lambda_j\}_j$ form a lacunary sequence and $2 + \tilde \varepsilon$ is the conjugate exponent of $2 - \varepsilon$. Now, by Remark~\ref{rmk: dual de la transf} we have the following chain of continuous operators
\[
\overline{C^\infty_c(\rr)}^{L^{q'(\cdot)}}
\xrightarrow{\Fourier} L^{p'(\cdot)}(\mathbb{R})
\xrightarrow{\text{restr.}} L^{p'(\cdot)}[0,M]
\hookrightarrow L^{1}[0,M].
\]
Repeating the previous construction with $q'(\cdot)$ and $\tilde{\varepsilon}$ in place of $p(\cdot)$ and $\varepsilon$, respectively, we obtain, for all sufficiently large $m$, a function $f_m$ such that
\[
C\, m^{1/2} \le \|\hat{f}_m\|_{L^1[0,M]}
\quad \text{and} \quad
\|f_m\|_{L^{q'(\cdot)}(E)} \le D\, m^{1/(2+\tilde\varepsilon)}.
\]
A closer inspection of the construction of $f_m$ shows that it is a continuous function with compact support contained in $E$. In particular, $f_m$ can be approximated in $L^{q'(\cdot)}(\rr)$ by functions in $C^\infty_c(\rr)$ with support in $E$. Hence, $f_m \in \overline{C^\infty_c(\rr)}^{L^{q'(\cdot)}}$, which gives the desired contradiction.
\end{proof}

\begin{remark}
    Observe that the inclusion assumptions
    $$
    \{x\in\rr : p(x) > 2 + \varepsilon\}\subseteq G\qquad \text{and}\qquad \{x\in\rr : q(x) < 2 - \varepsilon\}\subseteq U
    $$
    in the previous theorem, are made only for convenience. Indeed, if the conclusion holds for every measurable superset of
\(
\{x\in\mathbb{R}: p(x)>2+\varepsilon\}
\)
whose excess has finite measure, then it also holds for any measurable set differing from it by a set of finite measure (since the latter is contained in such a superset). We have chosen the present formulation because it keeps the focus on the upper and lower level sets $\{x\in\mathbb{R}: p(x)>2+\varepsilon\}$ and $\{x\in\mathbb{R}: q(x)<2-\varepsilon\}$,
which are the primary objects of our study.

\end{remark}

\section{Open Problems}\label{problems}
The results obtained in this paper naturally lead to several open problems and directions for future research. We conclude by highlighting some of these questions.
\begin{enumerate}[label=\arabic*.]
\item \label{problem: equivalence} In light of Theorems~\ref{thm: necessary and sufficient thm} and \ref{Thm: equivalencia q_+=infinito}, it is natural to ask the following question.
Let $p(\cdot), q(\cdot) \in \mathcal{P}(\rr)$ be such that the Fourier transform $\Fourier: L^{p(\cdot)}(\rr)\to L^{q(\cdot)}(\rr)$ is bounded, and let $r(\cdot), s(\cdot)$ be the defect exponents defined by
$$
\frac{1}{r(\cdot)}=\frac{1}{p_-}-\frac{1}{p(\cdot)} \quad \text{and} \quad \frac{1}{s(\cdot)}=\frac{1}{q(\cdot)}-\frac{1}{(p_-)'}.
$$
Does $1\in L^{r(\cdot)}(\rr)\cap L^{s(\cdot)}(\rr)$?

One may also naturally ask whether the boundedness of $\Fourier: L^{p(\cdot)}(\rr)\to L^{q(\cdot)}(\rr)$ is equivalent to condition $(iii)$ of Theorem~\ref{thm: necessary and sufficient thm}. This question will be addressed in a forthcoming paper, where we prove that condition $(iii)$ is not sufficient for the boundedness of the Fourier transform. The proof relies on some auxiliary results that are developed there.

\item Related to the previous question, albeit more tractable, is the following problem.  Assume that $p(\cdot)$ is continuous, even, and non-decreasing on $(0,+\infty)$. Can one find necessary and sufficient conditions for the boundedness of $\Fourier: L^{p(\cdot)}(\rr)\to L^{p'(\cdot)}(\rr)$ in terms of the rate at which $p(x)$ converges to $p_-$ as $|x|\to+\infty$? In view of Corollary~\ref{cor:p- and LHinfty imply F bounded}, we know that $1/p(\cdot) \in \LH_\infty$ is sufficient for the boundedness of $\Fourier: L^{p(\cdot)}(\rr)\to L^{p'(\cdot)}(\rr)$ in this case. However, to the best of our knowledge, a necessary and sufficient condition is not yet known.

\item\label{problem: sets} In Theorem~\ref{thm:lagunares} we proved that if $\Fourier: L^{p(\cdot)}(\rr)\to L^{q(\cdot)}(\rr)$ is bounded, then the set $\{x:p(x)>2+\varepsilon\}$ cannot contain infinitely many intervals of a given length. The following question was our original motivation. Let $\alpha>p_-$. Does the set $$\{x\in \rr: p(x)>\alpha\}$$ have finite measure? And if $\alpha = 2+\varepsilon$? Analogous questions can be posed for the set $\{x\in \rr: q(x)<\beta\}$, where either $\beta <q_+$ or $\beta=2-\varepsilon$.

\item Throughout this work, we restrict our attention to the one-dimensional Fourier transform. It is natural to ask whether results analogous to those obtained in Theorems~\ref{thm: necessary and sufficient thm} and \ref{Thm: equivalencia q_+=infinito} can be established in the $n$-dimensional setting. We expect that, at least for sufficient conditions for the boundedness of the Fourier transform, one should impose conditions on
\[
\pinfmrtheta := \lim_{M \to +\infty} \essinf_{r \ge M} p(re^{i\theta}) \quad \text{and}\quad \qinfMrtheta := \lim_{M \to +\infty} \esssup_{r \ge M} q(re^{i\theta})
\]
for all $\theta\in [0,2\pi)$.

\end{enumerate}

\raggedbottom
\printbibliography[title={REFERENCES}]
\end{document}